\documentclass[11pt,reqno]{amsart}

\usepackage{amssymb,amsmath,epsfig}
\usepackage{amsfonts}
\usepackage{pst-plot}

\newcommand{\Z}{\mathbb{Z}}

\newcommand{\C}{\mathbb{C}}

\newcommand{\SL}{{\text {\rm SL}}}
\def\H{\mathbb{H}}
\newcommand{\leg}[2]{\genfrac{(}{)}{}{}{#1}{#2}}
\newcommand{\ov}{\overline}
\newcommand{\cal}{\mathcal}
\newtheorem{theorem}{Theorem}[section]

\newtheorem{corollary}[theorem]{Corollary}
\newtheorem{proposition}[theorem]{Proposition}

\newtheorem*{theorem*}{Theorem}
\newtheorem{remark}[theorem]{Remark}

\numberwithin{equation}{section}

\title{Rank and crank moments for overpartitions}

\date{\today}
\author{Kathrin Bringmann, Jeremy Lovejoy, and Robert Osburn}
\address{School of Mathematics, University of Minnesota, Minneapolis, MN 55455, U. S. A.}

\address{CNRS, LIAFA, Universit\'e Denis Diderot,
2, Place Jussieu, Case 7014, F-75251 Paris Cedex 05, FRANCE}

\address{School of Mathematical Sciences, University College Dublin, Belfield, Dublin 4, Ireland}

\email{brigman@math.umn.edu}

\email{lovejoy@liafa.jussieu.fr}

\email{robert.osburn@ucd.ie}

\thanks{The first author was partially supported by NSF grant DMS-0757907 while
 the second and third authors were partially supported by a PHC
Ulysses grant. Part of this paper was written while the first author was in residence at the
Max-Planck Institute. She thanks the institute for providing a stimulating environment.}

\subjclass[2000]{Primary: 11F11, 11P83; Secondary: 05A19}

\begin{document}
\begin{abstract}
We study two types of crank moments and two types of rank moments
for overpartitions.  We show that the crank moments and their
derivatives, along with certain linear combinations of the rank
moments and their derivatives, can be written in terms of
quasimodular forms. We then use this fact to prove exact relations
involving the moments as well as congruence properties modulo $3$,
$5$, and $7$ for some combinatorial functions which may be
expressed in terms of the second moments.  Finally, we establish a
congruence modulo $3$ involving one such combinatorial function
and the Hurwitz class number $H(n)$.

\end{abstract}

\maketitle

\section{Introduction}
Dyson's rank of a partition is the largest part minus the number
of parts \cite{dyson}.
Let $N(m,n)$ denote the number of partitions of $n$ whose rank is
$m$.
The Andrews-Garvan crank is either the
largest part, if $1$ does not occur, or the difference between the
number of parts larger than the number of $1$'s and the number of
$1$'s, if $1$ does occur \cite{AG}.
  For $n \neq 1$ let $M(m,n)$ denote the number of partitions
of $n$ whose crank is $m$. Even though there is only one partition
of one, for technical reasons we set $M(0,1) = -1$, $M(-1,1) =
M(1,1) = 1$, and $M(m,1) = 0$ otherwise.  Then the {\it $k$th rank
moment} $N_k(n)$ and the {\it $k$th crank moment} $M_k(n)$ are given by
\begin{equation}
N_k(n) := \sum_{m \in \mathbb{Z}} m^k N(m,n),
\end{equation}
and
\begin{equation}
M_k(n) :=  \sum_{m \in \mathbb{Z}} m^k M(m,n).
\end{equation}

Since their introduction by Atkin and Garvan \cite{atkgar}, the
rank and crank moments and their linear combinations have been the
subject of a number of works \cite{An1,An2,Br1,bgm,Fo-On1,Ga1}.
A key role in several of these studies is played by the fact that
the crank moments and their derivatives, along with a specific
linear combination of the rank moments and their derivatives, can
be expressed in terms of quasimodular forms. Here we shall see
that this holds in the case of overpartitions as well.

Recall that an \textit{overpartition} \cite{Co-Lo1} is a partition in which
the first occurrence of each distinct number may be overlined. For
example, the $14$ overpartitions of $4$ are
\begin{equation} \label{pbarof4}
\begin{gathered}
4, \overline{4}, 3+1, \overline{3} + 1, 3 + \overline{1},
\overline{3} + \overline{1}, 2+2, \overline{2}
+ 2, 2+1+1, \overline{2} + 1 + 1, 2+ \overline{1} + 1, \\
\overline{2} + \overline{1} + 1, 1+1+1+1, \overline{1} + 1 + 1 +1.
\end{gathered}
\end{equation}
We denote by $\ov{P}$ the generating function for overpartitions
(throughout $q=e^{2 \pi i \tau}$ and $\tau=x + iy$ with $x$, $y \in \mathbb{R}$)
\cite{Co-Lo1},
$$
\ov{P}=\ov{P} (q)= \prod_{n \geq 1} \frac{(1+q^n)}{(1-q^n)}.
$$

The case of overpartitions is somewhat different from that of
partitions.  First, there are two distinct ranks of interest:
Dyson's rank and the $M2$-rank \cite{Lo1}.  The $M2$-rank is a bit
more complicated than Dyson's rank.  We use the notation
$\ell(\cdot)$ to denote the largest part of an object, $n(\cdot)$
to denote the number of parts, and $\lambda_o$ for the
subpartition of an overpartition consisting of the odd
non-overlined parts.  Then the $M2$-rank of an overpartition
$\lambda$ is
$$
\text{$M2$-rank}(\lambda) := \bigg \lceil \frac{\ell(\lambda)}{2}
\bigg \rceil - n(\lambda) + n(\lambda_o) - \chi(\lambda),
$$
where $\chi(\lambda) = 1$ if the largest part of $\lambda$ is odd
and non-overlined and $\chi(\lambda) = 0$ otherwise.

Let $\ov{N}(m,n)$ (resp. $\ov{N2}(m,n)$) denote the number of
overpartitions of $n$ whose rank (resp. $M2$-rank) is $m$.  We
define the rank moments $\ov{N}_k(n)$ and $\ov{N2}_k(n)$, along
with their generating functions $\ov{R}_k$ and $\ov{R2}_k$, by
\begin{equation}
\ov{R}_k=\ov{R}_k (q):= \sum_{n \geq 0}\ov{N}_k(n)q^n :=  \sum_{n \geq
0}\left( \sum_{m \in \mathbb{Z}} m^k \ov{N}(m,n)\right) q^n,
\end{equation}
and
\begin{equation}
\ov{R2}_k=\ov{R2}_k (q):= \sum_{n \geq 0}\ov{N2}_k(n)q^n :=  \sum_{n \geq 0}
\left( \sum_{m \in \mathbb{Z}} m^k \ov{N2}(m,n) \right) q^n.
\end{equation}
We note that in light of the symmetries $\ov{N}(m,n) =
\ov{N}(-m,n)$ \cite{Lo.5} and $\ov{N2}(m,n) = \ov{N2}(m,n)$
\cite{Lo1}, we have $\ov{R}_k= \ov{R2}_k = 0$ when $k$ is odd.

The second difference between partitions and overpartitions is
that in the latter case no notion of crank has been defined.
Indeed, the crank for partitions arose because of its relation to
Ramanujan's congruences, and Choi has shown that no such
congruences exist for overpartitions \cite{Ch1}.  What we will be
required to consider are two ``residual cranks".  The first
residual crank of an overpartition is obtained by taking the crank
of the subpartition consisting of the non-overlined parts.  The
second residual crank is obtained by taking the crank of the
subpartition consisting of all of the even non-overlined parts
divided by two.

Let $\ov{M}(m,n)$ (resp. $\ov{M2}(m,n)$) denote the number of
overpartitions of $n$ with first (resp. second) residual crank
equal to $m$.  Here we make the appropriate modifications based on
the fact that for partitions we have $M(0,1) = -1$ and $M(-1,1) =
M(1,1) = 1$.  For example, the overpartition $\ov{7} + \ov{5} +
\ov{2} + 1$ contributes a $-1$ to the count of $\ov{M}(0,15)$ and
a $+1$ to $\ov{M}(-1,15)$ and $\ov{M}(1,15)$.  Define the crank
moments $\ov{M}_k(n)$ and $\ov{M2}_k(n)$, along with their
generating functions $\ov{C}_k$ and $\ov{C2}_k$, by
\begin{equation}
\ov{C}_k=\ov{C}_k(q) := \sum_{n \geq 0} \ov{M}_k(n)q^n :=  \sum_{n \geq 0}
\left( \sum_{m \in \mathbb{Z}} m^k \ov{M}(m,n) \right) q^n,
\end{equation}
and
\begin{equation}
\ov{C2}_k=\ov{C2}_k(q) := \sum_{n \geq 0} \ov{M2}_k(n) q^n :=  \sum_{n \geq 0}
\left(\sum_{m \in \mathbb{Z}} m^k \ov{M2}(m,n)\right) q^n.
\end{equation}
As with the rank moments, the crank moments turn out to be $0$ for
$k$ odd (see \eqref{Cofzq} and \eqref{calCofzq}).

We are now ready to state the quasimodularity properties of the
rank and crank moments for overpartitions.

\begin{theorem} \label{main1}
For $k \geq 1$ let $\ov{\cal{W}}_{k}$ denote the space of
quasimodular forms on $\Gamma_0(2)$ of weight at most $2k$ having
no constant term. The following functions are in $\ov{P} \cdot
\ov{\cal{W}}_{k}$:
\begin{itemize}
\item[$(i)$]
The functions in
$$
\ov{\cal{C}}_k := \left\{\delta_q^m\left(\ov{C}_{2j}\right) : m \geq 0,\,  1 \leq j \leq k, \,
 j+m
\leq k \right\},
$$
\item[$(ii)$]
 the functions in
$$
\ov{\cal{C}2}_k := \left\{\delta_q^m\left(\ov{C2}_{2j}\right) : m \geq 0,\,
1 \leq j \leq k\, ,
j+m \leq k\right\},
$$
\item[$(iii)$]
for $a = 2k$ the function
$$
\begin{gathered}
\left(a^2-3a+2\right)\overline{R}_a +
 2\sum_{i=1}^{a/2-1} {a \choose
2i}\left(3^{2i} - 2^{2i} - 1\right)\delta_q\ov{R}_{a-2i} \\ +
\sum_{i=1}^{a/2-1} \left({a \choose 2i}(2^{2i}+1) + 2{a \choose
2i+1}\left(1-2^{2i+1}\right)
+ \frac{1}{2}{a \choose 2i+2}
\left(3^{2i+2} -
2^{2i+2} - 1\right) \right) \ov{R}_{a-2i},
\end{gathered}
$$
\item[$(iv)$]
for $a=2k$ the function
$$
\begin{gathered}
\left(a^2-3a+2\right)\overline{R2}_a + \frac{1}{2}\sum_{i=1}^{a/2-1} {a
\choose 2i}
\left(3^{2i} - 2^{2i} - 1\right)\delta_q\ov{R2}_{a-2i} \\ +
\sum_{i=1}^{a/2-1} \left({a \choose 2i}(2^{2i}+1) + 2{a \choose
2i+1}\left(1-2^{2i+1}\right) + \frac{1}{2}{a \choose 2i+2}
\left(3^{2i+2} -
2^{2i+2} - 1\right) \right) \ov{R2}_{a-2i}.
\end{gathered}
$$
\end{itemize}
\end{theorem}

It turns out that for $k=2,3$, and $4$ the number of functions
above exceeds the dimension of $\ov{\cal{W}}_{k}$, which implies
relations among these functions.  In Corollaries
\ref{cor1}--\ref{cor3}, we compute several such relations.  This
is the same approach taken by Atkin and Garvan in their study of
rank and crank moments of partitions \cite{atkgar}.

Then we show how Theorem \ref{main1} can be used to deduce
congruence properties for combinatorial functions which can be
expressed in terms of second rank and crank moments. First, let
$nov(n)$ (resp. $ov(n)$) denote the sum, over all overpartitions
of $n$, of the non-overlined (resp. overlined) parts.  For
example, \eqref{pbarof4} shows that $ov(4) = 21$ and $nov(4) =
35$. The generating functions of $nov(n)$ and $ov(n)$ are given by (see Section 4)
\begin{eqnarray}
Nov(q)&:=& \sum_{n=0}^{\infty} nov(n)\, q^n= \ov{P}
\sum_{n=1}^{\infty} \frac{n\, q^n}{1-q^n}, \label{Nov} \\
Ov(q)&:=& \sum_{n=0}^{\infty} ov(n)\, q^n
= \ov{P}
\sum_{n=1}^{\infty} \frac{n\, q^n}{1+q^n}. \label{Ov}
\end{eqnarray}

\begin{theorem} \label{main2} We have
\begin{equation} \label{cong2}
(n+2)nov(n) \equiv (n^2+4n+3)ov(n) \pmod{5},
\end{equation}
and
\begin{equation} \label{cong3}
(n^2+1)nov(n) \equiv (4n^3-n^2-1)ov(n) \pmod{7}.
\end{equation}







\end{theorem}
Notice that congruences like \eqref{cong2} and \eqref{cong3} imply
simpler
congruences in arithmetic progressions for $ov(n)$ and $nov(n)$
modulo $5$ and $7$.

Next, let $\ov{spt1}(n)$ (resp. $\ov{spt2}(n)$) denote the sum,
over all overpartitions $\lambda$ of $n$, of the number of
occurrences of the smallest part of $\lambda$, provided this
smallest part is odd (resp. even).  Let $\ov{spt}(n)$ be the sum
of these two functions.  For example, using \eqref{pbarof4} we
have $\ov{spt1}(4) = 20$, $\ov{spt2}(4) = 6$, and $\ov{spt}(4) =
26$. When the overpartition has no overlined parts, $\ov{spt}(n)$
reduces to Andrews' smallest parts function $spt(n)$
\cite{An2,Fo-On1,Ga1}. The functions $\overline{spt2}(n)$ and
$\overline{spt}(n)$ can be easily computed using  (\ref{this}) and
(\ref{this2}).


\begin{theorem} \label{main3}
We have
\begin{equation} \label{cong4}
\ov{spt2}(3n) \equiv \ov{spt2}(3n+1) \equiv 0 \pmod{3},
\end{equation}
\begin{equation} \label{cong5}
\ov{spt}(3n) \equiv 0 \pmod{3},
\end{equation}
\begin{equation} \label{cong6}
\ov{spt2}(5n+3) \equiv 0 \pmod{5},
\end{equation}
and
\begin{equation} \label{cong7}
\ov{spt1}(5n) \equiv 0 \pmod{5}.
\end{equation}
\end{theorem}

To finish we use a different method to give a congruence modulo
$3$ between $\ov{spt1}(n)$ and $\overline{\alpha}(n)$, the number
of overpartitions with even rank minus the number with odd rank.
\begin{theorem} \label{main4}
We have
$$
\overline{spt1}(n) \equiv \leg{n}{3} \overline{ \alpha }(n) \pmod
3.
$$
\end{theorem}\begin{remark}
In \cite{Br-Lo2}, the coefficients $\overline{\alpha}(n)$ are related to
the Hurwitz class number $H(n)$ of binary quadratic forms of discriminant $-n$.
To be more precise, it is shown in \cite{Br-Lo2} that
\begin{equation} \label{formula1eq}
(-1)^n\ov{\alpha}(n) =
\begin{cases}
-4H(4n) & \text{if }n \equiv 1,2 \pmod{4}, \\
-24H(n) & \text{if }n \equiv 3 \pmod{8}, \\
-16H(n)&\text{if } n \equiv 7\pmod{8}, \\
-16H(n) - \frac{1}{3}r(n/4)&\text{if } 4 \mid n,
\end{cases}
\end{equation}
where $r(n)$ is given by
$$
\sum_{n=0}^{\infty} r(n)\, q^n :=\Theta^3(\tau),
$$
with $\Theta(\tau):=\sum_{n \in \Z} q^{n^2}$.
It is well-known that
\begin{equation} \label{rofn}
r(n) =
\begin{cases}
12H(4n)& \text{if $n \equiv 1,2 \pmod{4}$}, \\
24H(n)& \text{if $n \equiv 3 \pmod{8}$}, \\
r(n/4)& \text{if $n \equiv 0 \pmod{4}$}, \\
0& \text{if $n \equiv 7 \pmod{8}$},
\end{cases}
\end{equation}
thus modulo $3$,  $\overline{spt1}(n)$ is  related to class numbers.
\end{remark}
As a corollary, class number relations imply the following
multiplicative formula:
\begin{corollary} \label{HeckeTheorem}
Let $\ell \not=2,3$ be a prime. Then we have
$$
\overline{spt1}\left( \ell^2n\right) + \leg{-n}{\ell}
\overline{spt1}(n) +\ell \, \overline{spt1} \left(
\frac{n}{\ell^2}\right) \equiv (\ell+1) \overline{spt1}(n) \pmod
3.
$$
\end{corollary}
\begin{remark}
Work of the authors \cite{Br-Lo-O} shows that the generating function for
$\overline{spt1}(n)$ can  (up to a quasimodular form) be viewed as the
holomorphic part of a harmonic  Maass form (see Section \ref{HeckeCong}
for the definition). Corollary \ref{HeckeTheorem} now says that modulo $3$
the generating function for $\overline{spt1}(n)$ is a Hecke eigenform.
\end{remark}

The paper is organized as follows. In Section 2, we recall some
facts about quasimodular forms and prove Theorem \ref{main1}. In
Section 3, we compute some exact relations involving rank and
crank moments.  In Section 4, we write the combinatorial functions
in Theorems \ref{main2} and \ref{main3} in terms of the rank and crank
moments and prove these theorems.  In Section 5, we recall the
notion of harmonic Maass forms  along with some results from
\cite{Br-Lo-O}, and prove Theorem \ref{main4} and Corollary
\ref{HeckeTheorem}.

\section{Proof of Theorem \ref{main1}}
Before proving Theorem \ref{main1}, we recall a few facts about
quasimodular forms \cite{kz}.  First, quasimodular forms on
$\Gamma_0(N)$ may be regarded as polynomials in the Eisenstein
series $E_2$ whose coefficients are modular forms (of non-negative
weight) on $\Gamma_0(N)$. The reader unfamiliar with the theory
of modular forms may consult \cite{cmbs}. Here we have
\begin{equation*}
E_2 (\tau):= 1-24\sum_{n \geq 1} \frac{nq^n}{(1-q^n)}.
\end{equation*}
Second, the space of quasimodular forms on $\Gamma_0(N)$ is
preserved by the differential operator $\delta_q := q
\frac{d}{dq}$.  More specifically, this operator sends a
quasimodular form of weight $2k$ to a quasimodular form of weight
$2k+2$.  Finally, replacing $q$ by $q^2$ sends a quasimodular form
of weight $2k$ on $\Gamma_0(N)$ to a quasimodular form of weight
$2k$ on $\Gamma_0(2N)$.

\begin{proof}[Proof of Theorem \ref{main1}]

We now prove parts $(i)$ and $(ii)$ of Theorem \ref{main1}.  Let
$C(z;q)$ denote the two-variable generating function for the crank
of a partition,
\begin{equation*}
C(z;q) := \sum_{m \in \mathbb{Z} \atop n \geq 0} M(m,n)z^mq^n =
\frac{(q;q)_{\infty}}{(zq;q)_{\infty}(q/z;q)_{\infty}}.
\end{equation*}
Here we employ the standard $q$-series notation,
\begin{equation*}
(a;q)_{\infty} := \prod_{k \geq 0} \left(1-aq^k\right).
\end{equation*}
By definition, the residual cranks have two-variable generating
functions
\begin{equation} \label{Cofzq}
\overline{C}(z;q) := \sum_{m \in \mathbb{Z} \atop n \geq 0}
\overline{M}(m,n)z^mq^n = (-q;q)_{\infty}C(z;q) =
\frac{(q^2;q^2)_{\infty}}{(zq;q)_{\infty}(q/z;q)_{\infty}},
\end{equation}
and
\begin{equation} \label{calCofzq}
\overline{{C2}}(z;q) := \sum_{m \in \mathbb{Z} \atop n \geq 0}
\overline{{M2}}(m,n)z^mq^n =
\frac{(-q;q)_{\infty}}{\left(q;q^2\right)_{\infty}}
C\left(z;q^2\right) =
\frac{(-q;q)_{\infty}(q^2;q^2)_{\infty}}{(q;q^2)_{\infty}(zq^2;q^2)_{\infty}(q^2/z;q^2)_{\infty}}.
\end{equation}

Now using the differential operator $\delta_z := z\frac{d}{dz}$ we
have
\begin{equation*}
\delta_{z}^{j}\left(\overline{C}(z;q)\right) \Bigl \lvert_{z=1}=
\left\{ \begin{array}{ll} \overline{C}_{j} & \qquad \text{if } j \hspace{.1in}  \textup{is even}, \vspace{.05in} \\
 0 & \qquad \text{if } j \hspace{.1in} \textup{is odd}, \end{array}
\right.
\end{equation*}
and
\begin{equation*}
\delta_{z}^{j}\left(\overline{C2}(z;q)\right) \Bigl \lvert_{z=1}=
\left\{ \begin{array}{ll} \overline{C2}_{j} & \qquad \text{if } j \hspace{.1in}  \textup{is even}, \vspace{.05in} \\
0 & \qquad\text{if }  j \hspace{.1in} \textup{is odd}. \end{array}
\right.
\end{equation*}
But $\delta_{z}^{j}\left(\overline{C}(z;q)\right) =
(-q;q)_{\infty}\delta_{z}^{j}\left({C}(z;q)\right)$ and Atkin and Garvan
\cite[Section 4]{atkgar} have already shown that if $j \geq 1$,
then $\delta_{z}^{j}\left({C}(z;q)\right) \lvert_{z=1}$ is in the space
$P\cdot \cal{W}_j$, where $P=P(q) := 1/(q;q)_{\infty}$ is the generating
function for partitions and $\cal{W}_j$ is the space of
quasimodular forms of weight at most $2j$ on $\SL_2(\Z)$ having
no constant term. Since $\ov{P} = (-q;q)_{\infty}P$, we have that
$\overline{C}_{2j}$ is in $\ov{P}\cdot\ov{\cal{W}}_j$.  In a similar
way we see that $\ov{C2}_{2j}$ is in $\ov{P}\cdot\ov{\cal{W}}_j$.

To finish we may calculate that
\begin{equation*}
\delta_q\left(\ov{P} \right) = \ov{P}
\left(\sum_{n \geq
1}\frac{2nq^n}{(1-q^n)} - \sum_{n \geq
1}\frac{2nq^{2n}}{(1-q^{2n})}\right),
\end{equation*}
and hence $\delta_q(\ov{P}) \in \ov{P} \cdot \ov{\cal{W}}_1$.  By the
Leibniz rule and the fact that $\delta_q$ maps the space
$\ov{\cal{W}}_{k}$ into $\ov{\cal{W}}_{k+1}$, we have that $\delta_q
f \in \ov{P}\cdot\ov{\cal{W}}_{k+1}$ if $f \in
\ov{P}\cdot\ov{\cal{W}}_{k}$.  This completes the proof of parts
$(i)$ and $(ii)$.

For parts $(iii)$ and $(iv)$, we use partial differential
equations established in \cite{Br-Lo-O}.  Let $\ov{R}(z;q)$ denote
the two-variable generating function for $\ov{N}(m,n)$,
\begin{equation*}
\ov{R}(z;q) := \sum_{m \in \mathbb{Z} \atop n \geq 0}
\ov{N}(m,n)z^mq^n.
\end{equation*}
Thus we have
\begin{equation*}
\delta_{z}^{j}\left(\overline{R}(z;q)\right) \Bigl \lvert_{z=1}=
\left\{ \begin{array}{ll} \overline{R}_{j}
& \qquad \text{if } j \hspace{.1in}  \textup{is even}, \vspace{.05in} \\
0 & \qquad \text{if }  j \hspace{.1in} \textup{is odd}. \end{array}
\right.
\end{equation*}

We have the following partial differential equation which relates
$C(z;q)$ and $\ov{R}(z;q)$ \cite{Br-Lo-O}:
\begin{equation} \label{pdenostar}
\begin{aligned}
z(1+z)\frac{(q)_{\infty}^2}{(-q)_{\infty}} & \left[C(z;q)\right]^3 (-zq;q)_{\infty}\left(-q/z;q\right)_{\infty} \\
& = \Bigl ( 2(1-z)^2 (1+z) \delta_{q} + z(1+z) + 2z(1-z)
\delta_{z} +
\frac{1}{2} (1+z) (1-z)^2 \delta_{z}^2  \Bigr ) \ov{R}(z;q). \\
\end{aligned}
\end{equation}

Let $a$ be even and positive. After applying $\delta_{z}^{a}$ to
both sides of (\ref{pdenostar}) and setting $z=1$ we get
\begin{equation} \label{relation}
\begin{gathered}
\frac{1}{P\ov{P}} \sum_{j = 0}^a {a \choose j}
\delta_z^j\left\{(z^2+z)C(z;q)^3 \right\}\delta_z^{a-j}\{(-zq;q)_{\infty}\left(-q/z;q\right)_{\infty}\}
|_{z=1}
- \left(2^a+1\right)\ov{P} \\
- 2\left(3^a-2^a-1\right)\delta_q(\ov{P}) =
\left(a^2-3a+2\right)\overline{R}_a + 2\sum_{i=1}^{a/2-1} {a \choose
2i}\left(3^{2i} - 2^{2i} - 1\right)\delta_q\ov{R}_{a-2i} \\
+ \sum_{i=1}^{a/2-1} \left({a \choose 2i}(2^{2i}+1) + 2{a \choose
2i+1}\left(1-2^{2i+1}\right) + \frac{1}{2}{a \choose 2i+2}
\left(3^{2i+2} -
2^{2i+2} - 1\right)\right) \ov{R}_{a-2i}.
\end{gathered}
\end{equation}

We claim that the left hand side of (\ref{relation}) is in
$\ov{P}\cdot \ov{\cal{W}}_{k}$, where $2k=a$.  This is clearly
true for the final term.  For the first term, we have already
noted that for $j \geq 1$ we have $\delta_z^jC(z;q) \lvert_{z=1}
\in P\cdot\ov{\cal{W}}_j$.  As for
$(-zq;q)_{\infty}\left(-q/z;q\right)_{\infty}$, we may compute that

$$
\begin{aligned}
\delta_{z} \Bigl( (-zq;q)_{\infty}\left(-q/z;q\right)_{\infty} \Bigr) & = \Biggl (z
\sum_{m=1}^{\infty}
\frac{q^m}{1+ zq^m} - z^{-1} \sum_{m=1}^{\infty} \frac{q^m}{1+ z^{-1}q^m} \Biggr ) (-zq;q)_{\infty}\left(-q/z;q\right)_{\infty} \\
& = \Biggl (\sum_{m=1}^{\infty} \sum_{s=1}^{\infty}
(-1)^s q^{ms} (z^{-s} - z^{s}) \Biggr ) (-zq;q)_{\infty}\left(-q/z;q\right)_{\infty},\\
\end{aligned}
$$
and for $j \geq 1$,
$$
\left.
\delta_{z}^{j}\left(\sum_{m=1}^{\infty} \sum_{s=1}^{\infty} (-1)^s
q^{ms} \left(z^{-s} - z^{s}\right)\right) \right|_{z=1}=
\left\{ \begin{array}{ll} 0 & \qquad \text{if }j \hspace{.1in}  \textup{is even}, \vspace{.05in} \\
\displaystyle -2\sum_{m=1}^{\infty} \sum_{s=1}^{\infty} (-1)^s
s^{j} q^{ms} & \qquad \text{if }j \hspace{.1in} \textup{is odd}. \end{array}
\right.
$$

\noindent Then one can check that

$$
-2\sum_{m=1}^{\infty} \sum_{s=1}^{\infty} (-1)^s s^{j} q^{ms}
=-2^{j+2} \sum_{n \geq 1} \frac{n^jq^{2n}}{(1-q^{2n})} + 2\sum_{n
\geq 1} \frac{n^jq^{n}}{(1-q^{n})}.
$$
Thus for $j \geq 1$ we have
$$
\delta_z^j
\left.
\left\{(-zq;q)_{\infty}\left(-q/z;q\right)_{\infty}\right\} \right\lvert_{z=1} \in
(\overline{P}^2/P^2) \cdot \ov{\cal{W}}_j.
$$
Putting everything
together we see that the only contribution from the first term on
the left hand side which is not in $\ov{P}\cdot\ov{\cal{W}}_{k}$
is
$$\left. \frac{1}{P\ov{P}}\delta_z^a\left\{(z^2+z)\right\}C(z;q)^3(-zq;q)_{\infty}\left(-q/z;q\right)_{\infty}
\right\lvert_{z=1}.
$$
But this cancels with the second term on the left hand side.  This
establishes part $(iii)$.

The proof of part $(iv)$ is the same, except that we use the
partial differential equation \cite{Br-Lo-O}
\begin{equation*}
\begin{aligned}
2z(1+z) & \left(q^2; q^2\right)_{\infty}^2 \left[C\left(z;q^2\right)\right]^3 (-zq;q)_{\infty}\left(-q/z;q\right)_{\infty} \\
& = \Bigl ((1+z)(1-z)^2 \delta_{q} + 2z(1+z) + 4z(1-z)\delta_{z} +
(1+z)(1-z)^2 \delta_{z}^2 \Bigr) \ov{R2}(z;q). \\
\end{aligned}
\end{equation*}
Here $\ov{R2}(z;q)$ is the two-variable generating function for
$\ov{N2}(m,n)$, so that
\begin{equation*}
\delta_{z}^{j}\left(\overline{R2}(z;q)\right) \Bigl \lvert_{z=1}=
\left\{ \begin{array}{ll} \overline{R2}_{j} & \qquad \text{if }  j \hspace{.1in}  \textup{is even}, \vspace{.05in} \\
0 & \qquad \text{if }  j \hspace{.1in} \textup{is odd}. \end{array}
\right.
\end{equation*}

\end{proof}

\section{Exact relations}
From part $(b)$ of Proposition 1 in \cite{kz} and known formulas
for the dimensions of spaces of modular forms on $\Gamma_{0}(2)$
(see \cite{cmbs}), we have that the sequence
$\{\text{dim}\hspace{.025in} (\ov{\cal{W}_{k}})\}_{k =1}^{\infty}$ begins
$\{2,6,12,21,33,49,\dots\}$.  Suppose first that $k=2$.  Then
there are $6$ functions in parts $(i)$ and $(ii)$ of Theorem
\ref{main1}.  Computation (with MAPLE, for example) shows that
they are linearly independent.  Hence, each function in parts
$(iii)$ and $(iv)$ may be written as a linear combination of these
$6$ functions, and we compute the following:
\begin{corollary} \label{cor1}
We have
\begin{equation} \label{k2one}
\overline{N}_4(n) = (-8n-1)\ov{N}_2(n) +
\left(\frac{-216+24n}{77}\right)\overline{M}_2(n) \\ +
\frac{192}{77}\overline{M}_4(n) +
\left(\frac{260+184n}{77}\right)\ov{{M2}}_2(n) -
\frac{40}{11}\ov{{M2}}_4(n)
\end{equation}
and
\begin{equation} \label{k2two}
\overline{N2}_4(n) = (-2n-1)\ov{N2}_2(n) +
\left(\frac{-27+3n}{77}\right)\overline{M}_2(n) +
\frac{24}{77}\overline{M}_4(n) +
\left(\frac{71-131n}{77}\right)\ov{{M2}}_2(n) -
\frac{16}{11}\ov{{M2}}_4(n).
\end{equation}
\end{corollary}
Now let $k=3$.  Again we find that the $12$ functions in parts
$(i)$ and $(ii)$ of Theorem \ref{main1} are linearly independent,
and so the functions in parts $(iii)$ and $(iv)$ may be written in
terms of them.
Following the lead of Atkin and Garvan, we use \eqref{k2one} and
\eqref{k2two} to eliminate $\ov{N}_4(n)$ and $\ov{N2}_4(n)$, thus
expressing $\ov{N}_{6}(n)$ (resp. $\ov{N2}_6(n)$) in terms of
$\ov{N}_2(n)$ (resp. $\ov{N2}_2(n)$) and the crank moments.
\begin{corollary} \label{cor2}
We have
\begin{equation} \label{k3onebis}
\begin{aligned}
\ov{N}_6(n) &= \left(3+20n+48n^2\right)\ov{N}_2(n) +
\left(\frac{2192796}{274505} + \frac{123276n}{7595} +
\frac{-5185344n^2}{1921535}\right)\ov{M}_2(n) \\
&+ \left(\frac{-445728}{54901} +
\frac{-5730048n}{384307}\right)\ov{M}_4(n) +
\left(\frac{5376}{3565}\right)\ov{M}_6(n) \\ &+
\left(\frac{-386988}{39215} + \frac{-54556468n}{1921535} +
\frac{-30679392n^2}{1921535}\right)\ov{M2}_2(n) \\
&+ \left(\frac{96204}{7843} +
\frac{1412352n}{54901}\right)\ov{M2}_4(n) +
\left(\frac{-9056}{3565}\right)\ov{M2}_6(n)
\end{aligned}
\end{equation}
and
\begin{equation} \label{k3twobis}
\begin{aligned}
\ov{N2}_6(n) &= \left(3+5n+3n^2\right)\ov{N2}_2(n) +
\left(\frac{249003}{274505} + \frac{36273n}{83545} +
\frac{-162042n^2}{1921535}\right)\ov{M}_2(n) \\
&+ \left(\frac{-46014}{54901} +
\frac{-179064n}{384307}\right)\ov{M}_4(n) +
\left(\frac{168}{3565}\right)\ov{M}_6(n) \\ &+
\left(\frac{-765123}{274505} + \frac{6826601n}{1921535} +
\frac{4805874n^2}{1921535}\right)\ov{M2}_2(n) \\
&+ \left(\frac{39102}{7843} +
\frac{44136n}{54901}\right)\ov{M2}_4(n) +
\left(\frac{-3848}{3565}\right)\ov{M2}_6(n).
\end{aligned}
\end{equation}
\end{corollary}

Now for $k=4$, there are $22$ functions in Theorem \ref{main1} and
the dimension of $\ov{P}\cdot\ov{\cal{W}}_k$ is $21$.  This
implies a relation among these $22$ functions.  If we would like
relations wherein only one type of rank moment occurs then we may
combine the function
$$F=F(q) :=
q(q;q)_{\infty}^6\left(q^2;q^2\right)_{\infty}^9 := \sum_{n \geq
1} a_F(n)q^n
$$
with the functions in $\ov{\cal{C}}_4$ and $\ov{\cal{C}2}_4$ to
get a basis for $\ov{P} \cdot \ov{\cal{W}}_4$.  (The fact that $F$
is in this space follows from the fact that
$q(q;q)_{\infty}^8\left(q^2;q^2\right)_{\infty}^8$ is a cusp form
of weight $8$ on $\Gamma_0(2)$).  Then each of the functions in
$(iii)$ and $(iv)$ of Theorem \ref{main1} may expressed in terms
of this basis.  We display the relation for the case of
$\ov{N}_k(n)$, again using results above to eliminate the $4$th
and $6$th rank moments in favor of the $2$nd.


\begin{corollary} \label{cor3}
\begin{equation} \label{k4one}
\begin{aligned}
& \ov{N}_8(n) = \left(-17-112n-224n^2 - 256n^3\right)\ov{N}_2(n) +
\left(\frac{15815680}{70153149}\right)a_F(n)\\ &+
\left(\frac{-3743678558672}{83365325395} +
\frac{-141447890442736n}{1750671833295} +
\frac{-135995781048448n^2}{1750671833295} + \frac{9269071448192n^3}{583557277765}\right)\ov{M}_2(n) \\
&+ \left(\frac{772193500416}{16673065079} +
\frac{9412063348224n}{116711455553}  +
\frac{9106119501824n^2}{116711455553}\right)\ov{M}_4(n) \\ &+
\left(\frac{-75923065344}{7578665945}  +
\frac{-737849634816n}{83365325395}\right)\ov{M}_6(n)
+ \left(\frac{2715648}{2125853}\right)\ov{M}_8(n) \\
&+ \left(\frac{4640559869932}{83365325395} +
\frac{260410320833296n}{1750671833295} +
\frac{345677277049024n^2}{1750671833295} + \frac{50935374262656n^3}{583557277765}\right)\ov{M2}_2(n) \\
&+ \left(\frac{-1173668372016}{16673065079} +
\frac{-2419446071808n}{16673065079}+
\frac{-2390306267136n^2}{16673065079}\right)\ov{M2}_4(n) \\ &+
\left(\frac{130253841984}{7578665945} +
\frac{1671243657216n}{83365325395}\right)\ov{M2}_6(n)
+ \left(\frac{-4858240}{2125853}\right)\ov{M2}_8(n). \\
\end{aligned}
\end{equation}
\end{corollary}

When $k\geq 5$, the number of functions in Theorem \ref{main1} is
smaller than the dimension of $\ov{P}\cdot\ov{\cal{W}}_k$.
Presumably this could be remedied by adding functions of the form
$\ov{P}f$, where $f$ is a cusp form, along with their $\delta_q$-
derivatives.  We shall not pursue this here.


\section{Proof of Theorems \ref{main2} and \ref{main3}}
We begin this section by expressing our combinatorial functions in
terms of the second moments $\ov{N}_2(n)$, $\ov{N2}_2(n)$,
$\ov{M}_2(n)$, and $\ov{M2}_2(n)$.
\begin{proposition} \label{prop1}
We have $nov(n) = \frac{1}{2}\ov{M}_2(n)$ and $ov(n) =
\frac{1}{2}\ov{M}_2(n) - \ov{M2}_2(n)$.
\end{proposition}
\begin{proof}
Dyson \cite{dyson1} has shown that $M_2(n) = 2np(n)$, where $p(n)$
is the number of partitions of $n$.  Since
$$
\sum_{n \geq 0} M_2(n)q^n = \left. \delta_z^2 C(z;q) \right\lvert_{z=1},
$$
we have that
$$
\begin{aligned}
\sum_{n \geq 0} \ov{M}_2(n)q^n &= (-q;q)_{\infty} \delta_z^2
C(z;q) \lvert_{z=1} \\
&= (-q;q)_{\infty} \sum_{n \geq 0} 2np(n)q^n \\
&= 2\sum_{n \geq 0} nov(n)q^n.
\end{aligned}
$$
Similarly, we find that $\ov{M2}_2(n)$ may be interpreted as
$enov(n)$, where $enov(n)$ denotes the sum, over all
overpartitions of $n$, of the even non-overlined parts.  Using
Euler's identity between the number of partitions of $n$ into odd
parts and the number of partitions of $n$ into distinct parts, we
see that $nov(n) - enov(n) = ov(n)$.
\end{proof}
Note that by applying $\delta_q$ to $P$, we see that
\begin{equation} \label{delqP}
\frac{1}{(q;q)_{\infty}}\sum_{n \geq 0} \frac{nq^n}{(1-q^n)} =
\sum_{n \geq 0} np(n)q^n,
\end{equation}
which gives equations \eqref{Nov} and \eqref{Ov}.

We now turn to the smallest parts functions.
\begin{proposition}
We have $\ov{spt}(n) = \ov{M}_2(n) - \ov{N}_2(n)$ and $\ov{spt2}(n) =
\ov{M2}_2(n) - \ov{N2}_2(n)$.
\end{proposition}
\begin{proof}
By the work in \cite{Br-Lo-O}, we find that
\begin{equation} \label{this}
\sum_{n=0}^{\infty}
\ov{spt}(n)q^n = \frac{(-q;q)_{\infty}}{(q;q)_{\infty}} \sum_{n \geq
1} \frac{2nq^n}{(1-q^n)} - \sum_{n=0}^{\infty} \ov{N}_2(n)q^n,
\end{equation}
and
\begin{equation} \label{this2}
\sum_{n=0}^{\infty}
\ov{spt2}(n)q^n = \frac{(-q;q)_{\infty}}{(q;q)_{\infty}} \sum_{n \geq
1} \frac{2nq^{2n}}{(1-q^{2n})} - \sum_{n=0}^{\infty} \ov{N2}_2(n)q^n.
\end{equation}
Combining \eqref{delqP} with \eqref{this}, \eqref{this2}, and the proof of Proposition
\ref{prop1} finishes the proof.
\end{proof}

We now prove the congruences in Theorems \ref{main2} and
\ref{main3}.

\begin{proof}[Proof of Theorem \ref{main2}]

For \eqref{cong2}, we simply multiply \eqref{k3onebis} by 5 and
reduce modulo $5$.  The result is
\begin{equation} \label{withh}
\left(2n^2+n+2\right)\ov{M}_2(n) + \left(n^2+4n+3\right)\ov{M2}_2(n) \equiv 0 \pmod{5},
\end{equation}
which implies the desired congruence.

For \eqref{cong3}, we first multiply \eqref{k2one} by $7$ and
reduce modulo $7$.  The result is
\begin{equation} \label{above}
(2+6n)\ov{M}_2(n) + 6\ov{M}_4(n) + (2+4n)\ov{M2}_2(n) \equiv 0
\pmod{7}.
\end{equation}
Next we take the set $\ov{\cal{C}}_4 \cup \ov{\cal{C}2}_4 \cup
\{F\}$ and replace $\delta_q\ov{C}_4$ by $\ov{C}_2\ov{C}_4/\ov{P}$
and $\delta_q^2\ov{C}_4$ by $\ov{C}_2\ov{C}_6/\ov{P}$.  This turns
out to be a basis for $\ov{P}\cdot\ov{\cal{W}}_4$.  Expressing the
function in part $(iii)$ of Theorem \ref{main1} in terms of this
basis, multiplying by $7$ and reducing modulo $7$ gives
$$
\left(4+6n+2n^2+3n^3\right)\ov{M}_2(n) + 6\ov{M}_4(n) +
\left(4n+5n^2+n^3\right)\ov{M2}_2(n) \equiv 0 \pmod{7}.
$$
Using \eqref{above} to substitute for $\ov{M}_4(n)$ gives
$$
\left(2n^3+3n^2+3\right)\overline{M2}_2(n) \equiv
\left(n^3+3n^2+3\right)\ov{M}_2(n) \pmod 7,
$$
and the congruence \eqref{cong3} follows.
\end{proof}

\begin{proof}[Proof of Theorem \ref{main3}]

First reduce \eqref{k2two} modulo $3$ to obtain
$$
(2n+2)\overline{N2}_2(n) \equiv (2n+2)\ov{M2}_2(n) \pmod{3}.
$$
Since $\ov{spt2}(n) = \ov{M2}_2(n) - \ov{N2}_2(n)$, we have
\eqref{cong4}.

Reducing \eqref{k2one} modulo $3$ we obtain
$$
(2n+2)\overline{N}_2(n) \equiv (2n+2)\ov{M2}_2(n) \pmod{3}.
$$
Combined with the fact that $nov(3n) \equiv - ov(3n) \pmod{3}$
(since $nov(n) + ov(n) = n\ov{p}(n)$) and the fact that
$\ov{spt}(n) = \ov{M}_2(n) - \ov{N}_2(n)$, we have \eqref{cong5}.

Next we perform the same computation used to obtain
\eqref{k3twobis}, except that we replace $\delta_q^2\ov{C2}_2$ by
$\ov{C}_2\ov{C}_4/\ov{P}$.  Reducing the result modulo $5$ gives
\begin{equation} \label{with2}
\left(1-n^2\right)\ov{N2}_2(n) \equiv \left(2n^2+3\right)\ov{M}_2(n) \pmod{5}.
\end{equation}
Combining this with \eqref{withh} when $n$ is replaced by $5n+3$
gives \eqref{cong6}.

Finally we perform the same calculation used to obtain
\eqref{k3onebis}, again replacing $\delta_q^2\ov{C2}_2$ by
$\ov{C}_2\ov{C}_4/\ov{P}$.  Reducing the result modulo $5$ gives
$$
\left(3+2n^2\right)\overline{N}_2(n) \equiv \left(n+4n^2\right)\ov{M}_2(n) +
(4+4n)\ov{M2}_2(n) \pmod 5.
$$
Combining this with \eqref{with2} and \eqref{withh} when $n$ is
replaced by $5n$, together with the fact that $\ov{spt1}(n) =
\ov{M}_2(n) - \ov{N}_2(n) - \ov{M2}_2(n) + \ov{N2}_2(n)$, gives
\eqref{cong7}.

\end{proof}


\section{Proof of Theorem \ref{main4} and Corollary \ref{HeckeTheorem}}
\begin{proof}[Proof of Theorem \ref{main4}] \label{HeckeCong}
Let $\ov{Spt1}=\ov{Spt1}(q)$ denote the generating function for
$\ov{spt1}(n)$ and let $\ov{f}=\ov{f}(q)$ denote the generating function
for $\ov{\alpha}(n)$.  Since by  \eqref{cong4} and \eqref{cong5}
we have
$$
\overline{spt1}(3n) \equiv 0 \pmod 3,
$$
to prove Theorem \ref{main4} it is enough to show that
$$
G=G(q):= \leg{  \bullet }{   3} \otimes \left(4 \overline{Spt1} -
\leg{   \bullet}{   3 }  \otimes \overline{f}  \right) \equiv 0
\pmod 3,
$$
where  for a character $\chi$ and a $q$-series $g$,  $\chi \otimes
g$ denotes the twist of $g$ by
 $\chi$, i.e., the $n$th Fourier coefficient of $g$ is multiplied by $\chi(n)$.
Let us next recall the definition of a harmonic   Maass form.
If $k\in \frac{1}{2}\Z\setminus
\Z$, then the weight $k$ hyperbolic
Laplacian is given by
\begin{equation}\label{laplacian}
\Delta_k := -y^2\left( \frac{\partial^2}{\partial x^2} +
\frac{\partial^2}{\partial y^2}\right) + iky\left(
\frac{\partial}{\partial x}+i \frac{\partial}{\partial y}\right).
\end{equation}
If $v$ is odd, then define $\epsilon_v$ by
\begin{equation}
\epsilon_v:=\begin{cases} 1 \ \ \ \ &{\text {\rm if}}\ v\equiv
1\pmod 4,\\
i \ \ \ \ &{\text {\rm if}}\ v\equiv 3\pmod 4. \end{cases}
\end{equation}
Moreover we let $\chi$ be a Dirichlet character.
 A {\it harmonic Maass form of weight $k$ with Nebentypus $\chi$ on a subgroup
$\Gamma \subset \Gamma_0(4)$} is any smooth function $g:\H\to \C$
satisfying the following:
\begin{enumerate}
\item For all $A= \left(\begin{smallmatrix}a&b\\c&d
\end{smallmatrix} \right)\in \Gamma$ and all $\tau \in \H$, we
have
\begin{displaymath}
g(A\tau)= \leg{c}{d}^{2k}\epsilon_d^{-2k} \chi(d)\,(c\tau+d)^{k}\ g(\tau).
\end{displaymath}
\item We  have that $\Delta_k g=0$.
\item The function $g(\tau)$ has
at most linear exponential growth at all the cusps of $\Gamma$.
\end{enumerate}
Define the integral
\begin{equation*}
  \overline{NH}(\tau):= \frac{1}{2\sqrt{2}\pi i} \int_{-\bar \tau}^{i \infty}
   \frac{\eta^2(u)}{\eta(2u)(-i (\tau+u))^{\frac32}}
  \, d u,
\end{equation*}
\noindent where $\eta(\tau)$ is Dedekind's eta function.
Combining (\ref{this}) and (\ref{this2}) with Theorems 4.1 and  5.1 of
 \cite{Br-Lo-O}, we may conclude that
$$
\overline{\mathcal{M}}_1(\tau):= \overline{Spt1} +\frac{1}{12}
\frac{\eta(2\tau) }{\eta^2(\tau) } E_2\left(\tau \right) - \frac{1}{3}
\frac{\eta(2\tau) }{\eta^2(\tau) } E_2\left(2\tau\right)
 + \overline{NH}(\tau)
$$
is a weight $\frac32$  harmonic Maass form on $\Gamma_0(16)$.
From \cite{Br-Lo2} we have that
\begin{equation*}
\overline{\mathcal{M}}(\tau) := \overline{f}    -
4\overline{NH}(\tau)
\end{equation*}
is also a harmonic Maass form of weight $\frac32$ on
$\Gamma_0(16)$.

Turning back to the proof of  Theorem \ref{main4},
 it is not hard to check that
$$
G \equiv  H
  \pmod 3,
$$
where
\begin{multline*}
H=H(q):=
\leg{  \bullet }{   3} \otimes
\left(4
\left(
\overline{Spt1} +\frac{1}{12}
 \frac{\eta(2\tau) }{   \eta^2(\tau)}
E_2\left(\tau \right) - \frac{1}{3} \frac{ \eta(2 \tau)   }{
\eta^2(\tau)}
 E_2\left( 2\tau \right) \right)   \right.
\\
\left.
+\frac{\eta(2 \tau)}{\eta^2(\tau)}
+   \frac{ \eta(2 \tau)   }{  3 \eta^2(\tau)}
  \left(
 -
E_4\left(2\tau \right)  + E_4(\tau) \right) - \leg{   \bullet}{   3 }
\otimes \overline{f}
  \right).
\end{multline*}
As in the proof of Proposition 4.1 of \cite{BO}, one can show that
 the non-holomorphic parts of
$\overline{\mathcal{M}}_1(\tau)$ and $\overline{ \mathcal{M}}(\tau)$ are
supported on negative squares.
This easily yields that $H$ is a linear combination of
weakly holomorphic modular forms, i.e.~ meromorphic modular forms
whose poles are supported in the cusps,
 of weights  $-\frac12$,  $\frac32$, and $\frac72$ on
$\Gamma_0(144)$.
We next place bounds on the orders of vanishing of $H$ in the
cusps. Clearly $E_4(\tau)$  and $E_4(2\tau)$ have no poles. Moreover from
the transformation law of $\overline{f}$ (see \cite{Br-Lo2}) it follows that
$\overline{f}$ also has no poles. Using this one can show
 that poles can only arise from $\frac{ \eta(2\tau)}{ \eta^2(\tau)}$ and thus
 are of the form $\frac{a}{c}$ with $c$ odd. Using properties of twists,
we can bound the orders of vanishing of $H$ at $\frac{a}{c}$ with
$c$ odd as follows:
 If $9|c$, its order can be bounded by $-\frac{1}{16}$, if $3\parallel
 c$ its order is bounded by $-\frac{9}{16}$, and if $3 \nmid c$ the order is bounded by $-\frac{1}{144}$.
 This now easily yields that
$\frac{\eta^{ 18 }(\tau) }{\eta^9(2\tau) } G$ is the sum of three
holomorphic modular forms of weight $4$, $6$, and $8$,
respectively. Using the fact that $\frac{\eta^6(\tau) }{
\eta^2(3\tau)}$ is a holomorphic weight $2$ modular form on
$\Gamma_0(9)$ satisfying
$$
\frac{\eta^6(\tau) }{ \eta^2(3\tau)} \equiv 1 \pmod 3,
$$
we can show that $G$ is congruent  to a holomorphic
modular form of weight $8$ on $\Gamma_0(144)$ modulo $3$. Sturm's Theorem now  gives that this
form is  congruent to $0$ if the first
$$
\left[  \frac{8}{12} \left[  \SL_2(\Z) : \Gamma_0(144)   \right]
\right] +1 = 193
$$
coefficients are congruent $0$ modulo $3$. This can be done by
MAPLE.
\end{proof}
Corollary \ref{HeckeTheorem} now follows easily from Theorem \ref{main4} and the following
\begin{proposition} \label{ClassIdentity}
Let $\ell \not= 2$, $3$. Then we have
\begin{equation} \label{overIdentity}
\overline{\alpha} \left(\ell^2 n \right) + \leg{-n}{\ell}
\overline{  \alpha }(n) + \ell\,  \overline{\alpha } \left(
\frac{n}{ \ell^2}  \right) = ( \ell +1) \overline{\alpha}(n).
\end{equation}
\end{proposition}
\begin{proof}
To prove (\ref{overIdentity}), we have to show that
$$
g_ {\ell}(\tau) :=
 \overline{f}| T_{  \ell^2   }
- ( \ell +1)
 \overline{f} =0,
$$
where $T_{\ell}$ is the usual half-integral  weight
Hecke-operator. Using that $  \frac{ \eta^2(\tau)}{ \eta(2\tau)}$ is a
Hecke eigenform with eigenvalue $1+\frac{1}{\ell}$, one obtains from
\cite{BrO} that $g_{\ell}(\tau)$ is a weakly holomorphic modular form
of weight $\frac32$ on $\Gamma_0(16)$. Since the coefficients of
$\bar{f}$ have only polynomial growth it is a holomorphic form.
The valence formula now gives that $g_{\ell}=0$ if its first $4$
Fourier coefficients equal $0$.
 Thus to finish the proof, we have to show that (\ref{overIdentity}) is true for $0 \leq n \leq 3$.
 For $n=0$ this claim is trivial.
 For the other cases recall   (\ref{formula1eq}) and (\ref{rofn}).
Moreover we need the fact  \cite{Co} that
 if  $-n = Df^2$, where $D$ is a negative fundamental discriminant,
then
\begin{equation} \label{Cohenformula}
H(n) = \frac{h(D)}{w(D)} \sum_{d \mid f} \mu(d) \left( \frac{D}{d}
\right) \sigma_1(f/d).
\end{equation}
Here $h(D)$ is the class number of $\mathbb{Q}(\sqrt{D})$, $w(D)$
is half the number of units in the ring of integers of
$\mathbb{Q}(\sqrt{D})$, $\sigma_1(n)$ is the sum of the divisors
of $n$, and $\mu(n)$ is the M\"obius function. We only show
(\ref{overIdentity}) for $n=1$ since the   other cases follow
similarly. In this case we have to show that
$$
\overline{\alpha} \left(\ell^2 \right) = 2 \left( \ell+1 -
\leg{-1}{\ell} \right).
$$
Firstly we have from (\ref{formula1eq})  that
$$
\overline{\alpha} \left( \ell^2 \right) =4 H \left(4 \ell^2
\right).
$$
Since $h(-4)=1$ and $\omega(-4)=2$,
 (\ref{Cohenformula}) yields
$$
\overline{\alpha} \left(\ell^2 \right) = 2 \left( \sigma_1( \ell)
- \leg{-1}{\ell} \right) = 2 \left(\ell+1 - \leg{-1}{\ell}
\right),
$$
as claimed.
\end{proof}



\end{document}